\documentclass[12pt]{amsart}
\usepackage{amscd,amssymb}
\usepackage[plainpages,backref,urlcolor=blue]{hyperref}
\usepackage{tikz}
\usetikzlibrary{calc}
\theoremstyle{plain}
\newtheorem{thm}{Theorem}[section]

\newtheorem{prop}[thm]{Proposition}
\newtheorem{cor}[thm]{Corollary}
\newtheorem{conj}[thm]{Conjecture}
\newtheorem{case}{Case}

\theoremstyle{definition}
\newtheorem{definition}[thm]{Definition}
\newtheorem{ex}[thm]{Example}
\newtheorem{ntc}[thm]{Notation}

\theoremstyle{remark}
\newtheorem{rk}[thm]{Remark}

\numberwithin{equation}{section}
\newcommand\Z{{\mathbb Z}}
\newcommand{\C}{\mathbb{C}}
\newcommand{\PP}{\mathbb{P}}
\newcommand{\aff}{\ensuremath{\mathbb{A}}}
\DeclareMathOperator{\Bs}{Bs}

\DeclareMathOperator{\mdr}{mdr}
\DeclareMathOperator{\sing}{Sing}
\DeclareMathOperator{\ar}{AR}
\DeclareMathOperator{\dar}{ar}
\DeclareMathOperator\gal{Gal}

\newenvironment{romenum}
{\renewcommand{\theenumi}{\roman{enumi}}
\renewcommand{\labelenumi}{(\theenumi)}
\begin{enumerate}}{\end{enumerate}}

\newcommand\enet[1]{\renewcommand\theenumi{#1} 
\renewcommand\labelenumi{\theenumi}} 
\begin{document}

\date{\today}

\title[On some conjectures about free and nearly free divisors]{On some conjectures about free and nearly free divisors}
\author[E. Artal]{E. Artal Bartolo}
\address{IUMA\\
Departamento de Matem{\'a}ticas, Facultad de Ciencias, Universidad de Zara\-goza,
c/ Pedro Cerbuna 12, 50009 Zaragoza, SPAIN}
\email{artal@unizar.es}

\author{L. Gorrochategui}

\author{I. Luengo}

\author[A. Melle]{A. Melle-Hern{\'a}ndez}
\address{ICMAT (CSIC-UAM-UC3M-UCM) \\
Departamento de {\'A}lgebra, 
Facultad de Ciencias Matem{\'a}ticas, Universidad Complutense, 28040 Madrid,  SPAIN}
\email{leire.gg@gmail.com, iluengo@mat.ucm.es,amelle@mat.ucm.es}
\thanks{}



\dedicatory{Dedicated  to Gert-Martin Greuel on the occasion of his 70th birthday}

\thanks{The first author is partially supported by the Spanish grant MTM2013-45710-C02-01-P and
 Grupo Geometr{\'i}a of Gobierno de Arag{\'o}n/Fondo Social Europeo.
The last three authors are partially supported by the Spanish grant MTM2013-45710-C02-02-P}

\begin{abstract}
In this paper infinite families of examples of irreducible free  and nearly free curves  in the complex 
projective  plane which are not rational curves  and whose local singularites can have an arbitrary number of branches are given. 
All these examples answer negatively to 
some conjectures  proposed by A. Dimca and G.  Sticlaru.  Our examples  say nothing about 
the most remarkable conjecture  by A. Dimca and G.  Sticlaru, i.e. 
every rational cuspidal plane curve is either free or nearly free.
\end{abstract}

\maketitle

\section{Introduction}


The notion of free divisor was introduced by K.~Saito \cite{KS}
in the study of discriminants of versal unfoldings of germs of isolated hypersurface singularites.
Since then many interesting and unexpected applications to Singularity Theory and Algebraic Geometry 
has been appearing.  In this paper we are mainly focused on complex  projective  plane curves and 
we adapt the corresponding notions and results to this set-up.
The results contained in this paper have needed a lot of computations in order to get the right statements.
All of them have been done using
the computer  algebra system \texttt{Singular}~\cite{Sing} through \texttt{Sagemath}~\cite{sage}.  
We thanks \texttt{Singular}'s team for such a great mathematical tool and specially to
Gert-Martin for his dedication to \texttt{Singular} developpement.

Let $S:=\C[x,y,z]$ be the polynomial ring endowed with the natural graduation
$S=\bigoplus_{m=0}^\infty S_m$ by homogeneous polynomials.
Let $f\in S_d$ be a homogeneous polynomial of degree $d$ in the polynomial ring 
and let  $C\subset \PP^2$ defined by $f=0$ and assume that $C$ is reduced. 
We denote by $J_f$ the Jacobian ideal of $f$, i.e.
 the homogeneous ideal in $S$ spanned by $f_x,f_y,f_z$, and  by $M(f)=S/J_f$ 
the corresponding graded ring, called the Jacobian (or Milnor) algebra of $f$. 

Let $I_f$ be the saturation of the ideal $J_f$ with respect to the maximal ideal $(x,y,z)$ in $S$ and let $N(f)=I_f/J_f$ 
be the corresponding graded quotient. Recall that the curve $C: f=0$ is called a \emph{free divisor} if  $N((f)=  I_f/J_f =0,$ 
see e.g.~\cite{Sernesi}.

A. Dimca and G. Sticlaru introduced in \cite{DStNF} the notion of nearly free divisor which is a sligth modification of the notion of free
divisor.  The curve $C$ is called  \emph{nearly free divisor}   
if $N(f)\ne 0$ and $\dim N(f)_k \leq 1$ for any $k$.

The main results in \cite{DStFD,DStNF} and many infinite families of  examples motivate the following conjecture.

\begin{conj}\cite{DStNF}\label{q2}
\mbox{}

\begin{enumerate}
\enet{\rm(\roman{enumi})}
\item\label{c1i} Any rational cuspidal curve $C$ in the plane is either free or nearly free. 
\item\label{c1ii} An irreducible plane curve $C$ which is either free or nearly free is  rational.
\end{enumerate}
\end{conj}

In \cite{DStNF}, the authors  provide some interesting results supporting the statement of
Conjecture~\ref{q2}\ref{c1i};
in particular, Conjecture~\ref{q2}\ref{c1i}    holds for rational cuspidal curves of even degree~\cite[Theorem~4.1]{DStNF}. 
They  need a topological assumption on the cusps which is not 
fulfilled all the time when the degree is odd, see \cite[Theorem 4.1]{DStNF}.

They proved also 
that  this conjecture holds for a curve $C$ 
with an abelian fundamental group $\pi_1(\PP^2 \setminus C)$ or having as degree a prime power, 
see  \cite[Corollary~4.2]{DStNF} and the discussion in~\cite{ad:2015}. 
Moreover, any unicuspidal rational curve with a unique Puiseux pair is either free or nearly free, 
see \cite[Corollary~4.5]{DStNF}, except the curves of odd degree in one case of the classification of 
such unicuspidal curves obtained in  \cite{FLMN:London2006}.

As for Conjecture~\ref{q2}\ref{c1ii}, note that reducible nearly free curves
may have irreducible components which 
are not rational, see~\cite[Example~2.8]{DStNF}: a smooth cubic with three tangents
at aligned inflection points is nearly free (by the way, the condition of alignment can be removed,
at least in some examples computed
using~\cite{Sing}).
For free curves, examples can be found using \cite[Theorem~2.7]{val:15} 
e.g. $(x^3-y^3)(y^3-z^3)(x^3-z^3)(a x^3+b y^3+c z^3)$ for generic $a,b,c\in\C$ such that
$a+b+c=0$. The conjectures in \cite{val:15} give some candidate examples in
less degree; it is possible to check that $(y^2 z-x^3)(y^2 z -x^3-z^3)=0$ is free 
(also compued with~\cite{Sing}).
Dimca and Sticlaru also proposed the following conjecture.
\begin{conj}\cite{DStNF}
\label{q3}
\mbox{}

\begin{enumerate}
\enet{\rm(\roman{enumi})}
\item\label{c2i} Any free irreducible plane curve $C$ has only singularities with at most two branches. 
\item\label{c2ii} Any nearly free irreducible plane curve $C$ has only singularities with at most three branches. 
\end{enumerate}
\end{conj}

 In this paper we give some examples of irreducible free   and nearly free   curves  in the complex 
projective  plane which are not rational curves giving counterexamples to Conjecture~\ref{q2}\ref{c1ii}. 
In the same set up, some examples of  irreducible  free curve whose
two singular points has any odd number of branches giving counterexamples 
Conjecture~\ref{q3}\ref{c2i} and an irreducible nearly free curve 
with just one singular point which has 4 branches  giving counterexamples 
Conjecture~\ref{q3}\ref{c2ii} are provided too.

Section 2  is devoted to collect well known results in the theory of free divisors and nearly free divisors mainly form the original 
papers of  A. Dimca and G. Sticlaru  \cite{DStFD,DStNF}. 
Also a characterization for being  nearly-free reduced  plane curve from A. Dimca in~\cite{D15} is also recorded. 
This characterization is similar to the characterization of being free  by du Plessis and Wall in \cite{duPCTC}.  

From Section~\ref{irred-free} it can be deduced that, for every  odd integer $k\geq 1$, 
the irreducible plane curve   $C_{5 k}$ of degree $d=5k$ defined by
\begin{equation*}
C_{5 k}:\, \, f_{5k}: =(y^k z^k-x^{2k})^2 y^k-x^{5k}=0,
\end{equation*}
has   geometric genus equals $g(C_{5 k})=\frac{(k-1)(k-2)}{2}$,  it  has only two singular points and  the  number of branches 
of $C_{5k}$ at each of them is exactly   $k$,  and moreover it is a free divisor, 
see Theorem \ref{free-k-branches}. This  is a counterexample to both the free part of Conjecture~\ref{q2}\ref{c1ii} and 
Conjecture~\ref{q3}\ref{c2i}.

From Section~\ref{irred-nfree} it can also be deduced that, for any odd integer $k\geq 1$, 
the irreducible plane curve   $C_{4 k}$ of degree $d=4k$ defined by
\begin{equation*}
 C_{4 k}:\, \, f_{4k}:= (y^k z^k-x^{2k})^2 -x^{3k}y^k=0, 
\end{equation*}
 has   genus equals $g(C_{4 k})=\frac{(k-1)(k-2)}{2}$, its singular set  consists of two singular points  and 
the  number of branches of $C_{4k}$ at each of them  is  $k$ and now $C_{4k}$ is a nearly free divisor, 
see Theorem \ref{nf-free-k-branches}. This  is a counterexample to both the nearly free part of Conjecture~\ref{q2}\ref{c1ii} and 
Conjecture~\ref{q3}\ref{c2ii} too.

In the familes studied before the number of singular points of the curves  is exactly two.
In Section \ref{ellip-near}, we  are looking for 
curves  giving  a counterexample to the nearly free part of Conjecture~\ref{q2}\ref{c1ii}
with unbounded genus and number of singularities. 
In particular,  for
every odd integer $k\geq 1$,  the irreducible curve $C_{2 k}$ of degree $d=2k$ defined by
 $$ C_{2 k}:\, \, f_{2k}:= x^{2k}+y^{2k}+z^{2k}-2(x^ky^k+x^kz^k+y^kz^k)=0, $$
has genus equals $g(C_{2 k})=\frac{(k-1)(k-2)}{2}$, its  singular set  
 $\sing(C_{2 k})$ consists of exactly $3k$ singular points
of type $\mathbb{A}_{k-1}$  and  it is a nearly
free divisor, see Theorem \ref{nf-free-genus}.

One of the main tools to find such examples is the use of Kummer covers.
A Kummer cover is a map $\pi_k:\PP^2\to\PP^2$ given by 
$$
\pi_k([x:y:z]):=[x^k:y^k:z^k].
$$
 Since Kummer covers are
finite Galois unramified covers of $\PP^2\setminus \{xyz=0\}$ with $\gal(\pi_k)\cong\Z/k\Z \times\Z/k\Z$, Kummer covers are a very useful tool in order to construct complicated
algebraic curves starting from simple ones. We meanly refer to \cite[\S5]{aco:14} for a  systematic study of Kummer covers.

 In particular, these  familes of examples  
$\{C_{5 k}\}$ (which are free), $\{C_{4 k}\}$ and $\{C_{2 k}\}$ (which are nearly free)
are constructed as the pullback under the Kummer cover ~$\pi_k$
of the corresponding rational cuspidal curves: the quintic $C_5$ which is  a free divisor, and the corresponding nearly free divisors
defined by the cuartic $C_4$ and
 the  conic~$C_{2}$.

In Section \ref{nearly-4-branches}, an irreducible  nearly free curve of degree $49$  which  is  rational, it has  
   just one singular point which has 4 branches. It is a general element of the unique pencil
associated to any rational unicuspidal plane curve see \cite{Dai_Mell_14}.

\section{Free and nearly free plane curves after Dimca and Sticlaru} 

Let $S:=\C[x,y,z]$ be the polynomial ring endowed with the natural graduation
$S=\bigoplus_{m=0}^\infty S_m$ by homogeneous polynomials.
Let $f\in S_d$ be a homogeneous polynomial of degree $d$ in the polynomial ring 
Let $C$ be the plane curve in $\PP^2$ defined by $f=0$ and assume that $C$ is reduced. 
We have denoted by $J_f$ the Jacobian ideal of~$f$, i.e.
 the homogeneous ideal in $S$ spanned by $f_x,f_y,f_z$. Let $M(f)=S/J_f$ be 
the corresponding graded ring, called the Jacobian (or Milnor) algebra of $f$. 

The minimal degree of a Jacobian relation for $f$ is the integer $\mdr(f)$
defined to be the smallest integer $m\geq 0$ such that there is a nontrivial relation
\begin{equation}
\label{rel_m}
 af_x+bf_y+cf_z=0,\quad
(a,b,c)\in S_m^3\setminus(0,0,0).
\end{equation}
When $\mdr(f)=0$, then $C$ is a union of lines passing through one point, a situation easy to analyse. We assume from now on that 
$\mdr(f)\geq 1$.

\subsection{Free plane curves}
\mbox{}

We have denoted by $I_f$ the saturation of the ideal $J_f$ with respect to the maximal ideal $(x,y,z)$ in $S$.
Let $N(f)=I_f/J_f$ 
be the corresponding homogeneous quotient ring.

Consider the graded $S$-submodule 
$$
\ar(f)=\{(a,b,c)\in S^3\mid af_x+bf_y+cf_z=0\} \subset S^{3}
$$
 of {\it all relations} involving the derivatives of $f$, 
and denote by $\ar(f)_m$ its homogeneous part of degree~$m$.

\begin{ntc}
We set $\dar(f)_k=\dim\ar(f)_k$,   $m(f)_k=\dim M(f)_k$ and $n(f)_k=\dim N(f)_k$ for any integer $k$.  
\end{ntc}

We use the definition of freeness given by Dimca~\cite{D15}.

\begin{definition} The curve $C:f=0$ is a \emph{free divisor} if the following  equivalent conditions hold.

\begin{enumerate}

\item $N(f)=0$, i.e. the Jacobian ideal is saturated.

\item The minimal resolution of the Milnor algebra $M(f)$ has the following  form
$$0 \to S(-d_1-d+1) \oplus S(-d_2-d+1) \to S^3(-d+1) \xrightarrow{(f_x,f_y,f_z)}  S$$
for some positive integers $d_1, d_2$.
\item The graded $S$-module $\ar(f)$ is free of rank 2, i.e. there is an isomorphism 
$$\ar(f)=S(-d_1) \oplus S(-d_2)$$
for some positive integers $d_1, d_2$.
\end{enumerate}

\end{definition}
When $C$ is a free divisor, the integers $d_1 \leq d_2$ are called the {exponents} of $C$.  They satisfy the relations 
\begin{equation}
\label{free1}
 d_1+d_2=d-1 \text{ and } \tau(C)=(d-1)^2 - d_1d_2,
\end{equation}
where $\tau(C)$ is the total Tjurina number of $C$, see for instance \cite{DS14,DStFD}.
Following the deformations results in~\cite{Sernesi}, Sticlaru~\cite{Sti:15}
defines a curve $C\subset \PP^2$  to be  \emph{projectively rigid} if  $(I_f)_{d}= (J_f)_{d}$.
In particular, if $C$ is free then it is projectively rigid.

\begin{rk}
This notion of \emph{projectively rigid} differs from the classical one, see e.g. \cite{Flenner-Zaidenberg},
where a curve is projectively rigid if its equisingular moduli space is discrete. Note that
four lines passing through a point define a free divisor but its equisingular moduli space is defined
by the cross-ratio.
\end{rk}

\subsection{Nearly free plane curves}
\mbox{}

Dimca and Sticlaru introduced the notion of a nearly free divisor  which is more subtle, see \cite{DStNF},
or the curve version of \cite[Remark 5.2 and Theorem 5.3]{DStFS}.

\begin{definition} \cite{DStNF}
\label{def2}
The curve $C:f=0$ is a \emph{nearly free divisor} 
if the following  equivalent conditions hold.

\begin{enumerate}

\item $N(f) \ne 0$ and $n(f)_k \leq 1$ for any $k$.

\item The Milnor algebra $M(f)$ has a minimal resolution of the form
\begin{equation}
\label{r2}
 0 \to S(-d-d_2) \to S(-d-d_1+1) \oplus S^2(-d-d_2+1) \to S^3(-d+1) \xrightarrow{(f_0,f_1,f_2)}  S
\end{equation}
for some integers $1 \leq d_1 \leq d_2$, called the exponents of $C$.

\item There are 3 syzygies $\rho_1$, $\rho_2$, $\rho_3$ of degrees $d_1$, $d_2=d_3=d-d_1$ 
which form a minimal system of generators for the first-syzygy module $\ar(f)$.

\end{enumerate}

\end{definition}

If $C:f=0$ is nearly free, then the exponents $d_1 \leq d_2$ satisfy 
\begin{equation}
\label{nfree1}
 d_1+d_2=d \text{ and } \tau(C)=(d-1)^2-d_1(d_2-1)-1,
\end{equation}
see \cite{DStNF}. 
For both a free and a nearly free curve $C:f=0$, it is clear that $\mdr(f)=d_1$.

\begin{rk}\label{rk-near}
 In \cite{DStNF} it is shown that to construct  a resolution \eqref{r2} for a given polynomial $f$ 
one needs the integer $b:=d_2-d+2$ and the following ingredients.

\begin{romenum}
\item Three syzygies $r_i=(a_i,b_i,c_i) \in S^3_{d_i}$, $i=1,2,3$, for $(f_x,f_y,f_z)$, i.e. 
$$a_if_x+b_if_y+c_if_z=0,$$
necessary to construct the morphism 
$$\bigoplus_{i=1}^3S(-d_i-(d-1)) \to S^3(-d+1),
\quad (u_1,u_2,u_3) \mapsto u_1r_1+u_2r_2+u_3r_3.$$

\item One relation $R=(v_1,v_2,v_3) \in \bigoplus_{i=1}^3 S(-d_i-(d-1))_{b+2(d-1)} $ among $r_1,r_2,r_3$, i.e. $v_1r_1+v_2r_2+v_3r_3=0$, necessary to construct the morphism 
$$S(-b-2(d-1)) \to \bigoplus_{i=1,3}S(-d_i-(d-1))$$
by the formula $w \mapsto wR$. Note that $v_i \in S_{b-d_i+d-1}$.
\end{romenum}
\end{rk}

\begin{cor}\cite{DStNF}
\label{corC1}
Let $C:f=0$ be a nearly free curve of degree $d$ with exponents $(d_1,d_2)$. Then  $N(f)_k \ne 0$ for $ d+d_1-3 \leq k \leq d+d_2-3$ and $N(f)_k=0$ otherwise.
The curve $C$ is projectively rigid if and only if $d_1 \geq 4$. 
\end{cor}

\subsection{Characterization of free and nearly free reduced plane curves}
\mbox{}

Just recently  Dimca provides in \cite{D15} 
the following characterization of free and nearly free reduced  plane curves.
For a positive integer $r$, define
$$
\tau(r)_{\max}:= (d-1)(d-r-1)+r^2.
$$
\begin{thm}[\cite{D15}]\label{thm:dimca}
Let $C\subset\PP^2$ be a reduced curve of degree~$d$ defined by $f=0$, and let $r:=\mdr(f)$.
\begin{enumerate}
\enet{\rm(\arabic{enumi})}
\item\label{corCTC}
If $r=<\frac{d}{2}$, then $ \tau(C) =\tau(r)_{\max}$
if and only if $C:f=0$ is a free curve. 
\item\label{thm1-dimca} If $r=\leq\frac{d}{2}$, then $ \tau(C) = \tau(r)_{\max}-1$ if and only if $C$ is a nearly free curve.
\end{enumerate} 
\end{thm}

As it is recalled in~\cite{D15}, Theorem~\ref{thm:dimca}\ref{corCTC} 
is a corollary of \cite[Theorem~3.2]{duPCTC} by du Plessis and Wall.

\section{High-genus curves which are free or nearly free divisors}

\subsection{Transformations of curves by Kummer covers}\label{kummer}
\mbox{}

A Kummer cover is a map $\pi_k:\PP^2\to\PP^2$ given by 
$\pi_k([x:y:z]):=[x^k:y^k:z^k]$.
Kummer covers are a very useful tool in order to construct complicated
algebraic curves starting from simple ones. Since Kummer covers are
finite Galois unramified covers of $\PP^2\setminus \{xyz=0\}$ with $\gal(\pi_k)\cong\Z/k\Z \times\Z/k\Z$, 
topological properties of the new curves can be obtained: Alexander polynomial, fundamental group, 
characteristic varieties and so on (see 
\cite{ea:jag,ac:98,ji:99,uludag:01,Hirano-construction,ji-kloosterman,aco:14,LindnerDiplom} 
for papers using these techniques).

\begin{ex}
In \cite{uludag:01}, Uluda{\u{g}} constructs new examples of Zariski pairs using former ones and
Kummer covers. He also uses the same techniques to construct infinite families of curves
with finite non-abelian fundamental groups.
\end{ex}

\begin{ex}
In~\cite{Hirano-construction,ji:99}, the Kummer covers allow to construct curves with~\emph{many cusps} and
extremal properties for their Alexander invariants. These ideas are pushed further in \cite{ji-kloosterman}
where the authors find Zariski triples of curves of degree~$12$ with $32$~ordinary cusps (distinguished by their
Alexander polynomial). Within the same ideas Niels Lindner~\cite{LindnerDiplom}
constructed an example of a cuspidal curve $C'$ of degree 12 with 30 
cusps and Alexander polynomial $t^2-t+1$. For this, he started with a sextic $C_0$ with 6 cusps, 
admitting a toric decomposition. He pulled back $C_0$ under a Kummer map $\pi_2:\PP^2\to\PP^2$ 
ramified above three inflectional tangents of $C_0$. 
Since the sextic is of torus type, 
then same holds for the pullback. Lindner showed
that the Mordell-Weil lattice has rank 2 and that the Mordell-Weil group contains $A_2(2)$.
\end{ex}

A  systematic study of Kummer covers of projective plane curves have been done by
J.I.~Cogolludo ,  J.~Ortigas and the first named author  in \cite[\S5]{aco:14}. 
Some of  results are collected next.

Let ${C}$ be a (reduced) projective curve of degree~$d$ of equation $F_d(x,y,z)=0$
and let $\bar{C}_{k}$ be its transform by a Kummer cover $\pi_k$, $k\geq 1$. Note 
that $\bar{C}_{k}$ is a projective curve of degree~$d k$ of equation $F_d(x^k,y^k,z^k)=0$.

\begin{definition}\cite{aco:14}\label{typeofpoints}
Let $P\in\PP^2$ such that $P:=[x_0:y_0:z_0]$. We say
that $P$ is a point of \emph{type $(\C^*)^2$ } (or simply of \emph{type $2$}) if 
$x_0 y_0 z_0\neq 0$. If $x_0=0$ but $y_0 z_0\neq 0$ the point is said to be of 
\emph{type~$\C^*_x$} (types $\C^*_y$ and $\C^*_z$ are defined accordingly).
Such points will also be referred to as \emph{type $1$} points. The corresponding line
(either $L_X:=\{X=0\}$, $L_Y:=\{Y=0\}$, or $L_Z:=\{Z=0\}$) a type-$1$ lies on will be referred to as 
their \emph{axis}. The remaining points $P_x:=[1:0:0]$, $P_y:=[0:1:0]$, and 
$P_z:=[0:0:1]$ will be called \emph{vertices} (or type~0 points) and their axes are
the two lines (either $L_X$, $L_Y$, or $L_Z$) they lie on.
\end{definition}

\begin{rk}\label{todas-preimag}\cite{aco:14}
  Note that a point of type $\ell$, $\ell=0,1,2$ in $\PP^2$ has exactly $k^\ell$ 
preimages under~$\pi_k$. It is also clear that the local type of $\bar{C}_{ k}$ at any
two points on the same fiber are analytically equivalent. The singularities of $\bar{C}_{ k}$ are described in the following proposition.
\end{rk}

\begin{prop}\label{prop-sing}\cite{aco:14}
Let $P\in\PP^2$ be a point of type~$\ell$ and $Q\in \pi_k^{-1}(P)$. One has the following:
\begin{enumerate}
\enet{\rm(\arabic{enumi})}
 \item If $\ell=2$, then $({C},P)$ and $(\bar{C}_{k},Q)$ are analytically isomorphic. 
 \item If $\ell=1$, then $(\bar{C}_{ k},Q)$ is a singular point of type~$1$ if and only if~$m>1$,
where $m:=({C}\cdot\bar{L})_P$ and $\bar L$ is the axis of $P$.
 \item If $\ell=0$, then $(\bar{C}_{ k},Q)$ is a singular point.
\end{enumerate}
\end{prop}

\begin{rk}\label{todas-familias}
Using Proposition \ref{prop-sing} (1),  if Sing$(C)\subset \{xyz=0\}$ then Sing$(\bar{C}_{ k})\subset \{xyz=0\}$.
\end{rk}

\begin{ex}\label{ex-singpin}\cite{aco:14}
In some cases, we can be more explicit about the singularity type of $(\bar{C}_k,Q)$.
If $P$ is of type~$1$, $(C,P)$ is smooth and $m:=({C}\cdot\bar{L})_P$
then $(\bar{C}_k,Q)$ has the same topological type as $u_0^k-v_0^m=0$. In particular, if $m=2$,
then $(\bar{C}_k,Q)$ is of type $\mathbb{A}_{k-1}$.
\end{ex}

In order to better describe singular points of type $0$ and $1$ of $\bar C_k$ we will introduce some notation.
Let $P\in\PP^2$ be a point of type~$\ell=0,1$ and $Q\in \pi_k^{-1}(P)$ a singular point of 
$\bar C_k$. Denote by $\mu_P$ (resp. $\mu_Q$) the Milnor number of $C$ at $P$ (resp. $\bar C_k$ at~$Q$).
Since $\ell=0,1$, then $P$ and $Q$ belong to either exactly one or two axes. 
If $P$ and $Q$ belong to an axis $\bar L$, then $m_P^{\bar L}:=(C\cdot \bar L)_P$ 
(analogous notation for $Q$).
More specific details about singular points of types~$0$ and $1$ can be described as follows.

\begin{prop}\label{prop-sing-kummer}\cite{aco:14}
Under the above conditions and notation one has the following
\begin{enumerate}
\enet{\rm(\arabic{enumi})}
 \item\label{prop-sing-kummer1} 
For $\ell=1$, $P$ belongs to a unique axis $\bar L$ and
\begin{enumerate}
\item $\mu_Q=k\mu_P+(m_P^{\bar L}-1)(k-1)$,
\item If $(C,P)$ is locally irreducible
and $r:=\gcd(k,m_{P}^{\bar L})$, then $(C,Q)$
has  $r$ irreducible components
which are analytically isomorphic to each other.
\end{enumerate}

 \item\label{prop-sing-kummer0}
For $\ell=0$, $P$ belongs to exactly two axes $\bar L_1$ and $\bar L_2$ 
\begin{enumerate}
\item $\mu_Q=k^2 (\mu_P-1)+k(k-1)(m_P^{\bar L_1}+m_P^{\bar L_2})+1$ (There is a 
typo in the printed formula in~\cite{aco:14}: this one is obtained by adding $k-k^2$).
\item If $(C,P)$ is locally irreducible and $r:=\gcd(k,m_{P}^{\bar L_1},m_{P}^{\bar L_2})$,
then $(C,Q)$ has $k r$ irreducible components
which are analytically isomorphic to each other. 
\end{enumerate}
\end{enumerate}
\end{prop}

\subsection{Irreducible free curves with many  branches and high genus}\label{irred-free}
\mbox{}

Let us consider the quintic curve $C_5$, see Figure~\ref{fig:c5}, defined by $f_5:= (y z-x^2)^2 y-x^5=0$.
It has two singular points, $p_1=[0:1:0]$ of type $\mathbb{A}_4$ and
$p_2=[0:0:1]$ of type $\mathbb{E}_8$, so that it is  rational and cuspidal. 
This curve is free, see~\cite[Theorem~4.6]{DStFD}.
Let us consider the Kummer cover $\pi_k:\PP^2\to\PP^2$ given by 
$\pi_k([x:y:z]):=[x^k:y^k:z^k]$ and its Kummer
transform $C_{5 k}$, defined by $f_{5k}: =(y^k z^k-x^{2k})^2 y^k-x^{5k}=0$.

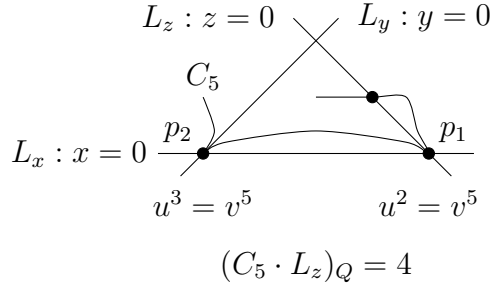
\begin{figure}[ht]
\begin{tikzpicture}[scale=1.5,vertice/.style={draw,circle,fill,minimum size=0.15cm,inner sep=0}]
\coordinate (A) at (-1,0);
\coordinate (B) at (1,0);
\coordinate (C) at (0,1);
\draw ($1.2*(A)-.2*(B)$)-- ($1.2*(B)-.2*(A)$);
\node[left] at  ($1.2*(A)-.2*(B)$) {$L_x:x=0$};
\draw ($1.2*(A)-.2*(C)$)-- ($1.2*(C)-.2*(A)$);
\node[left=20pt] at ($1.2*(C)-.2*(A)$) {$L_z:z=0$};
\draw ($1.2*(B)-.2*(C)$)-- ($1.2*(C)-.2*(B)$);
\node[right=20pt] at ($1.2*(C)-.2*(B)$) {$L_y:y=0$};
\node[vertice] at (A) {};
\node[vertice] at (B) {};
\node[vertice] at ($.5*(B)+.5*(C)$) {};
\draw plot [smooth] coordinates {(-1,.5) (-.9,.2) (A)};
\draw plot [smooth] coordinates {(A) (-.8,.1) (0,.2) (.8,.1) (B) (.9,.2) (.8,.5)
 ($.5*(B)+.5*(C)$) (0,.5)};
\node[below=10pt] at (A.south) {$u^3=v^5$};
\node[below=10pt] at (B.south) {$u^2=v^5$};
\node[above left] at (A) {$p_2$};
\node[above right] at (B) {$p_1$};
\node[above=20pt] at (A) {$C_5$};
\node at (0,-1) {$(C_5\cdot L_z)_Q=4$};
\end{tikzpicture}
\caption{Curve $C_5$}
\label{fig:c5}
\end{figure}

\begin{thm}\label{free-k-branches}
For any $k\geq 1$,  the curve $C_{5 k}$ of degree $d=5k$ defined by
\begin{equation}\label{eq:f5k}
C_{5 k}:\, \, f_{5k}: =(y^k z^k-x^{2k})^2 y^k-x^{5k}=0,
\end{equation}
verifies  the following properties: 
\begin{enumerate}
\enet{\rm(\arabic{enumi})}
\item\label{free-k-branches-1} $\sing(C_{5 k})=\{p_1, p_2\}$. The number 
of branches of $C_{5k}$ at  $p_2$ is  $k$, and at $p_1$, it equals $k$ (if $k$ is odd)
or $2k$ (if $k$ is even).
\item\label{free-k-branches-2} $C_{5 k}$ is a free divisor with exponents $d_1=2k$, $d_2=3k-1$ and $\tau(C_{5 k})=19k^2-8k+1$.
\item\label{free-k-branches-3} $C_{5 k}$ has two irreducible components 
of genus $\frac{(k-2)^2}{4}$ if $k$ is even and irreducible of genus $\frac{(k-1)(k-2)}{2}$ otherwise.
\end{enumerate}

\end{thm}

\begin{proof}
Part \ref{free-k-branches-1} is an easy consequence of 
in \cite[Lemma 5.3, Proposition 5.4 and Proposition 5.6]{aco:14}.
The singularites Sing$(C_5)=\{p_1,p_2\}$ are of type $0$, in the sense of the Kummer cover $\pi_k$ (see Definition \ref{typeofpoints}) 
and  $C_5$ has no singularities 
outside the intersection points of the axes. Moreover $C_5$
intersects the line $L_z$ transversally at a point of type $1$; 
then by Proposition \ref{prop-sing}~(2) and by
Remark~\ref{todas-familias},  the singularities of $C_{5 k}$
 are exactly  the points $p_1$ and $p_2$.

Since $p_1 $ and $p_2$ are of type $0$ we deduce the structure of $C_{5 k}$ at these points.
using  Proposition \ref{prop-sing-kummer}~\ref{prop-sing-kummer0}~(b).  At $p_1$ one has $(C_5, L_z)_{p_1}=5,$  
$(C_5, L_x)_{p_1}=2$ and $r_{p_1}=\gcd(k,2,5)=1$ for all $k$ , and so that the 
number of branches of $C_{5k}$ at $p_1$ is equal to $k$.
In the same way, at $p_2$,  the intersection  $(C_5, L_x)_{p_2}=2,$  
$(C_5, L_y)_{p_2}=4$ and $r_{p_2}=\gcd(k,2,4)=\gcd(k,2) $.
Ik $k$ is odd,  $r_{p_2}=1$ and the number of branches of $C_{5k}$ at $p_2$ is equal to $k$.
Otherwise $r_{p_2}=2$ and the number of branches of $C_{5k}$ at $p_2$ is equal to $2k$.

In order to prove~\ref{free-k-branches-2}, we follow the ideas of \cite[Theorem~4.6]{DStFD}.
Let us study first the syzygies of the free curve $C_5$. 
Let us denote by $D_{u,v,w}$, the diagonal matrix with entries $u,v,w$,
and define the vectors
$$
R_1=\left(0,\,2 y,\,x^{2} - 3 y z\right),
\quad R_2=\left(2(x^{2} -  y z),2( 5x^{2} - 4 x y + 15 y z),8 x - 45z\right).
$$
Let us denote by $J$ the Jacobian ideal $J$ of~$f_5$.
Let us denote by $J_x$ the ideal generated by $(x {f_5}_x,{f_5}_y,{f_5}_z)$.
In the same way, we consider the ideals $J_y$, $J_z$, $J_{x y}$, $J_{x z}$, $J_{y z}$, $J_{x y z}$.
The Table~\ref{table:syz} shows bases for the syzygies of these ideals,
computed with \texttt{Singular}~\cite{Sing}.
\begin{table}[ht]
\begin{center}
\begin{tabular}{|c|c|c|}\hline
Ideal & First generator & Second generator \\\hline
$J$ & $R_1\cdot D_{1,y,1}$ & $R_2\cdot D_{1,1,z}$\\\hline
$J_x$ & $R_1\cdot D_{1,y,1}$ & $R_2\cdot D_{1,x,x z}$ \\\hline
$J_y$ & $R_1$ & $R_2\cdot D_{y,1,y z}$ \\\hline
$J_z$ & $R_1\cdot D_{1,y z,1}$ &  $R_2$ \\\hline
$J_{x y}$ & $R_1$ & $R_2\cdot D_{y,x,x y z}$  \\\hline
$J_{x z}$ & $R_1\cdot D_{1,y z,1}$ & $R_2\cdot D_{1,x,x}$  \\\hline
$J_{y z}$ & $R_1\cdot D_{1,z,1}$ &  $R_2\cdot D_{y,1,y}$ \\\hline
$J_{x y z}$ & $R_1\cdot D_{1,z,1}$ &  $R_2\cdot D_{y,x,x y}$ \\\hline
\end{tabular}
\vspace*{3mm}
\caption{Bases of syzygies}
\label{table:syz}
\end{center}
\end{table}
Note that
$$
{f_{5 k}}_x\!=\!k x^{k-1} {f_5}_x(x^{k},y^{k},z^{k}),\ {f_{5 k}}_y=k y^{k-1} {f_5}_y(x^{k},y^{k},z^{k}),
\ {f_{5 k}}_z=k z^{k-1} {f_5}_z(x^{k},y^{k},z^{k}).
$$
Let $S_k:=\C[x^k,y^k,z^k]$. We have a decomposition
\begin{equation}\label{eq:desc_k}
S=\bigoplus_{(i,j,l)\in\{0,\dots,k-1\}} x^i y^j z^l S_k.
\end{equation}
By construction, ${f_{5 k}}_x\in x^{k-1} S_k$, ${f_{5 k}}_y\in y^{k-1} S_k$
and ${f_{5 k}}_z\in z^{k-1} S_k$. Hence, in order to compute the
syzygies $(a,b,c)$ among the partial derivatives of $f_{5 k}$, we need to characterize
the triples $(a,b,c)$ such that each entry belongs to a factor of the 
decomposition~\eqref{eq:desc_k}.

Let us assume that $a\in x^{i_x} y^{j_x} z^{j_x} S_k$, $b\in x^{i_y} y^{j_y} z^{j_y} S_k$
and $c\in x^{i_z} y^{j_z} z^{j_z} S_k$. We deduce that
$$
i_x+k-1\equiv i_y\equiv i_z\bmod k\Longrightarrow i=i_y=i_z\text{ and } i_x=
\begin{cases}
i+1&\text{ if }i<k-1\\
0&\text{ if }i=k-1.
\end{cases}
$$
Analogous relations hold for the other indices. We distinguish four cases:
\begin{case}
$i=j=l=k-1$.
\end{case}
In this case $a(x,y,z)=y^{k-1} z^{k-1}\alpha(x^k,y^k,z^k)$, 
$b(x,y,z)=x^{k-1} z^{k-1}\beta(x^k,y^k,z^k)$ and  $c(x,y,z)= x^{k-1} y^{k-1}\gamma(x^k,y^k,z^k)$.
Hence $(\alpha,\beta,\gamma)$ is a syzygy for the partial derivatives of $f_5$.
We conclude that $(a,b,c)$ is a combination of:
$$
R_1(x^k,y^k,z^k)\cdot D_{1,x^{k-1} y^k z^{k-1} ,x^{k-1} y^{k-1}}=
x^{k-1} y^{k-1} R_1(x^k,y^k,z^k)\cdot D_{1, y z^{k-1} ,1}
$$ 
and 
$$
R_2(x^k,y^k,z^k)\cdot D_{y^{k-1} z^{k-1},x^{k-1} z^{k-1} ,x^{k-1} y^{k-1} z^k}=
z^{k-1} R_2(x^k,y^k,z^k)\cdot D_{y^{k-1},x^{k-1} ,x^{k-1}  y^{k-1} z}.
$$ 
Taking out common factors we get syzygies of degree~$2 k$ and $ 3 k-1$.

\begin{case}
$i<k-1$, $j=l=k-1$.
\end{case}
In this case $a(x,y,z)=x^{i+1} y^{k-1} z^{k-1}\alpha(x^k,y^k,z^k)$, 
$b(x,y,z)=x^{i} z^{k-1}\beta(x^k,y^k,z^k)$ and  $c(x,y,z)= x^{i} y^{k-1}\gamma(x^k,y^k,z^k)$.
Hence $(\alpha,\beta,\gamma)$ is a syzygy for the generators of the ideal $J_x$.
It is easily seen that we obtain combination of generators of the
above syzygies. The other cases are treated in the same way.

We conclude that $C_{5k}$ is free with
$d_1=\mdr(f_{5k})=2k$ and  $d_2=d-1-d_1=5k-1-2k=3k-1$.
By equation 
\eqref{free1} $\tau(C_{5k})=19k^2-8k+1$ for all $k$.

In order to prove \ref{free-k-branches-3}, we study the branched cover 
$\tilde{\pi}_k:\tilde{C}_{5k}\to\tilde{C}_5$ between the normalizations of the curves.
The monodromy of this map as an unramified cover of
$\PP^2\setminus\{x y z=0\}$ is determined by an epimorphism
$$
H_1(\PP^2\setminus\{x y z=0\};\mathbb{Z})\to\mathbb{Z}_k\times\mathbb{Z}_k=:G_k
$$
such that the meridians of the lines are sent to $a_x,a_y,a_z$, a system of generators
of $G_k$ such that $a_x+a_y+a_z=0$. Since the singularities of $C_5$ are locally
irreducible, then $C_5$ and $\tilde{C}_5$ are homeomorphic, and the covering
$\tilde{\pi}_k$ is determined
by the monodromy map 
$$
H_1(\tilde{C}_5\setminus\{x y z=0\};\mathbb{Z})\to\mathbb{Z}_k\times\mathbb{Z}_k=:G_k
$$
obtained by composing with the map defined by the inclusion.
Hence $\tilde{C}_5\setminus\{x y z=0\}$ is isomorphic to
$\PP^1\setminus\{\text{three points}\})$. The image of a meridian
corresponding to a point $P$ in the axes is given by
$$
m_P^{L_x} a_x+m_P^{L_y} a_y+m_P^{L_z} a_z.
$$
Hence, we obtain $a_z$ (the smooth point), $3 a_x+5 a_y$ (the $\mathbb{E}_8$-point) and $2 a_x+ 4 a_z$
(the $\mathbb{A}_4$-point). In terms of the basis $a_y,a_z$ they read as
$a_z,2 a_y-3 a_z,-2 a_y+2 a_z$, i.e., the monodromy group is generated by $2 a_y, a_z$. If $k$ is pair,
the monodromy group is of index~$2$ in~$G_k$, and hence $\tilde{C}_{5k}$ has two connected components,
while it is $G_k$ when $k$ is odd and  $\tilde{C}_{5k}$ is connected. These properties
give us the statement about the number of irreducible components.

The genus can be computed using the singularities of $C_{5k}$ or via Riemann-Hurwitz's formula. Note that
the covering $\tilde{\pi}_k$ is of degree~$k^2$ with three ramification points;
at $p_2$ and the smooth point in the axis we find~$k$ preimages, while at $p_1$ we find $k$ preimages if
$k$ is odd and $2k$ preimages if it is even, because of \ref{free-k-branches-1}. Hence, for $k$ odd
\begin{equation*}
\chi(\tilde{C}_{5k})=-k^2+3 k\Longrightarrow g(\tilde{C}_{5k})=\frac{(k-1)(k-2)}{2}
\end{equation*}
and for $k$ even, where $\tilde{C}_{5k}=\tilde{C}_{5k}^1\cup \tilde{C}_{5k}^2$,
\begin{equation*}
\chi(\tilde{C}_{5k})=-k^2+4 k\Longrightarrow g(\tilde{C}_{5k}^i)=
\frac{2-\frac{\chi(\tilde{C}_{5k})}{2}}{2}=
\frac{(k-2)^2}{4}.
\qedhere
\end{equation*}
\end{proof}

So, for odd $k\geq 3$, the curve $C_{5k}$ is an irreducible free curve of positive genus 
whose  singularities have $k$ branches each. 
This  is a counterexample to both the free part of Conjecture~\ref{q2}\ref{c1ii} and 
Conjecture~\ref{q3}\ref{c2i}.

\begin{rk}
Up to projective transformation,
there are two quintic curves with two singular points of type $\mathbb{A}_4$ and
$\mathbb{E}_8$. One is $C_5: (y z-x^2)^2 y-x^5=0$, which is free; there is another
one $D_5: g=y^3 z^2-x^5=0$ (the contact of the tangent line to the
$\mathbb{A}_4$-point distinguishes both curves).
The curve~$D_5$ is nearly free; it is easily seen that $\mdr(g)=1$. Since both singular points
are quasihomogeneous, $\tau(C)=12$, and we may apply Theorem~\ref{thm:dimca}\ref{thm1-dimca};
the pair ($C_5$,$\tilde{C}_5$) is a kind of counterexample
to Terao's conjecture~\cite[Conjecture~4.138]{OT} for irreducible divisors (with constant Tjurina number),
compare with~\cite{scTo:09}.
\end{rk}

\subsection{Irreducible nearly free curves with many branches and high genus}\label{irred-nfree}
\mbox{}

The quartic curve $C_4$ defined by $f_4:= (y z-x^2)^2 -x^3y=0$ has  two singular points, $p_1=[0:1:0]$ of type $\mathbb{A}_2$ and
$p_2=[0:0:1]$ of type $\mathbb{A}_4$, in particular it  is rational and cuspidal.
We will consider the Kummer
transform $C_{4 k}$, defined by $f_{4k}:= (y^k z^k-x^{2k})^2 -x^{3k}y^k=0$,  of the curve $C_4.$ 

\begin{thm}\label{nf-free-k-branches}
For any $k\geq 1$,  the curve $C_{4 k}$ of degree $d=4k$ defined by
 $$ C_{4 k}:\, \, f_{4k}:= (y^k z^k-x^{2k})^2 -x^{3k}y^k=0, $$
verifies  the following properties 
\begin{enumerate}
\enet{\rm(\arabic{enumi})}
\item\label{nf-free-k-branches-1} $\sing(C_{4 k})=\{p_1, p_2\}$. The number 
of branches of $C_{4k}$ at each $p_2$ is  $k$, and at $p_1$, it equals $k$ (if $k$ is odd)
or $2k$ (if $k$ is even).
\item\label{nf-free-k-branches-2} $C_{4 k}$ is a nearly
free divisor with exponents  $d_1=d_2=d_3=2k$ and $\tau(C_{4 k})=6 k (2k-1)$.
\item\label{nf-free-k-branches-3} $C_{4 k}$ has two irreducible components 
of genus $\frac{(k-2)^2}{4}$ if $k$ is even and irreducible of genus $\frac{(k-1)(k-2)}{2}$ otherwise.
\end{enumerate}

\end{thm}

\begin{proof}
Since Sing$(C_4)(=\{p_1,p_2\})$ are points of type $0$,  $C_4$ meets $\{xyz=0\}$ 
at three points $p_1, p_2$
and transversally at $p_3$ which is of type 1 then Sing$(C_{4 k})=\{p_1, p_2\}$
and to prove Part \ref{nf-free-k-branches-1}
its enough to find the number of branches of $C_{4 k}$ at these points
using  Proposition \ref{prop-sing-kummer}~\ref{prop-sing-kummer0}~(b).  At $p_1$ one has $(C_4, L_z)_{p_1}=3,$  
$(C_4, L_x)_{p_1}=2$ and $r_{p_1}=\gcd(k,2,3)=1$ for all $k$ , and so that the 
number of branches of $C_{4k}$ at $p_1$ is equal to $k$.
In the same way, at $p_2$,  the intersection  $(C_4, L_x)_{p_2}=2,$  
$(C_4, L_y)_{p_2}=4$ and $r_{p_2}=\gcd(k,2,4)=\gcd(k,2) $.
Ik $k$ is odd,  $r_{p_2}=1$ and the number of branches of $C_{4k}$ at $p_2$ is equal to $k$.
Otherwise $r_{p_2}=2$ and the number of branches of $C_{4k}$ at $p_2$ is equal to $2k$.

The proof of Part \ref{nf-free-k-branches-2} follows the same guidelines as in Theorem~\ref{free-k-branches}.
With the notations of that proof, a generator system for the syzygies of~$J$
(Jacobian ideal of $f_4$) is given by:
\begin{equation}
\begin{aligned}
R_1:=&\left(y (3 x - 4  z) ,\,3 y(4 x - 3 y),\,z(9 y-20 x)\right),\\
R_2:=&\left(- x(x+ 2 z),\,-4 x^{2} + 3 x y + 10 y z,\,-z(3 x + 10z)\right),\\
R_3:=&\left(x y,\,-3 y^{2},\,2 x^{2} + 3 y z\right).
\end{aligned}
\end{equation}
These syzygies satisfy the relation $x R_1 +3 y R_2+10 zR_3=0$.
And by Dimca Steclaru Remark (\ref{rk-near}) $C_{4}$ is a nearly
free divisor with exponents  $d_1=d_2=d_3=2$.

For the ideal $J_z$, we have a similar situation.
For the other ideals, their syzygy space is free of rank~$2$.
Using these results it is not hard to prove that the syzygies of $f_{4k}$
are generated by
\begin{equation*}
\begin{aligned}
R_{k,1}:=&\left(y^k (3 x^k - 4  z^k) ,\,3 x^{k-1}y(4 x^k - 3 y^k),\,x^{k-1}z(9 y^k-20 x^k)\right),\\
R_{k,2}:=&\left(-x y^{k-1} (x^k+ 2 z^k),\,-4 x^{2k} + 3 x^k y^k + 10 y^k z^k,\,-y^{k-1}z(3 x^k + 10z^k)\right),\\
R_{k,3}:=&\left(x y^k z^{k-1},\,-3 y^{k+1} z^{k-1},\,2 x^{2k} + 3 y^k z^k\right).
\end{aligned}
\end{equation*}
The results follow as in the proof of Theorem~\ref{free-k-branches}.

These syzygies satisfy the relation $x R_{k,1} +3 y R_{k,2}+10 zR_{k,3}=0$ and by Dimca Steclaru Remark (\ref{rk-near}) $C_{4k}$ is a nearly
free divisor with exponents  $d_1=d_2=d_3=2k$ and by equation 
\eqref{nfree1}  $\tau(C_{4 k})=6 k (2k-1)$.

The proof of Part \ref{nf-free-k-branches-3} follows the same ideas  as in Theorem~\ref{free-k-branches}  \ref{free-k-branches-3}. 
\end{proof}

So, for odd $k\geq 3$, the curve $C_{4k}$ is an irreducible nearly free curve of positive genus 
whose  singularities have $k$ branches each. 
This  is a counterexample to both the nearly free part of Conjecture~\ref{q2}\ref{c1ii} and 
Conjecture~\ref{q3}\ref{c2ii}.

\subsection{Positive genus nearly-free curves with many singularities.}\label{ellip-near}
\mbox{} 

Let us consider the conic $C_2$ given by $f_2 =x^2+y^2+z^2-2(xy+xz+yz)=0$.
This conic is tangent to three axes and it is very useful to produce interesting
curves using Kummer covers.

\begin{figure}[ht]
\begin{center}
\begin{tikzpicture}[line cap=round,line join=round,x=1.0cm,y=1.0cm,scale=.4]
\draw(1.02,2.26) circle (2.2497110925627757cm);
\draw (0.,9.391190912668357)-- (5.764926282338022,-2.2539601776544487);
\draw (1.7360913065284267,9.307373371088104)-- (-3.754752182990192,-2.8425781801871306);
\draw (-5.628535136720331,0.06903741060312854)-- (7.658386004935119,-0.048545962331876635);
\end{tikzpicture}
\caption{Conic $C_2$.}
\label{fig:c2}
\end{center}
\end{figure}
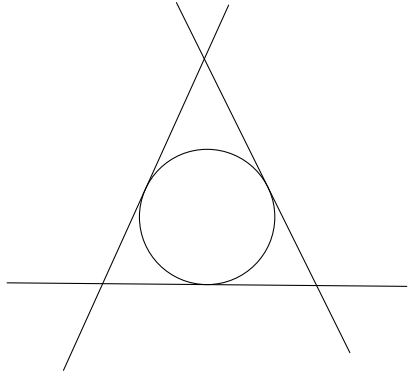

\begin{thm}\label{nf-free-genus}
For any $k\geq 1$,  the curve $C_{2 k}$ of degree $d=2k$ defined by
 $$ C_{2 k}:\, \, f_{2k}:= x^{2k}+y^{2k}+z^{2k}-2(x^ky^k+x^kz^k+y^kz^k)=0, $$
verifies  the following properties 
\begin{enumerate}
\enet{\rm(\arabic{enumi})}
\item\label{nf-free-genus-1} $\sing(C_{2 k})$ are $3k$ singular points
of type $\mathbb{A}_{k-1}$.
\item\label{nf-free-genus-2} $C_{2 k}$ is a nearly
free divisor with exponents  $d_1=d_2=d_3=k$ and $\tau(C_{2 k})=3k(k-1)$.
\item\label{nf-free-genus-3} $C_{2 k}$ is irreducible of genus~$\frac{(k-1)(k-2)}{2}$
if $k$ is odd and it has four irreducible smooth components of degree~$\frac{k}{2}$ if $k$ is even.
\end{enumerate}
\end{thm}

\begin{proof}
To prove \ref{nf-free-genus-1} it is enough to take into the account 
that $C_2$ is nonsingular and by Remark \ref{todas-familias} the singularites of $C_{2 k}$ satisfy
Sing$(C_{2k})\subset \{x y z=0\}$.
Moreover  $C_2$ is tangent to the 
three axes at 3 points $\{p_1, p_2,p_3\}$ of type $1$ with 
$(C_2, L_x)_{p_1}=(C_2, L_y)_{p_2}=(C_2, L_z)_{p_3}=2$ at these points. 
For $i=1, \ldots, 3,$ the points $p_i$ are  of type 1 and by Remark \ref{todas-preimag}  
all the $k$ preimages under $\pi_k$ are analytically equivalent  
 and by Example~\ref{ex-singpin} over each $p_i$ one has $k$ singular points of type $\mathbb{A}_{k-1}$.

Let us study \ref{nf-free-genus-2}.
A generator system for the syzygies of~$J$
(Jacobian ideal of $f_2$) is given by:
\begin{equation}
\begin{aligned}
R_1:=&\left(y -  z,\,y,\,- z\right),\\
R_2:=& \left(-x,\,z -  x,\, z\right),\\
R_3:=&\left(x,\,- y,\,x -  y\right)
.
\end{aligned}
\end{equation}
These syzygies satisfy the relation $xR_1+y R_2+zR_3=0$.
The other ideals have free $2$-rank syzygy modules.
An easy computation gives the following syzygies
for $f_{2k}$:
\begin{equation*}
\begin{aligned}
R_{k,1}:=&\left(y^k -  z^k,\,x^{k-1} y,\,- x^{k-1} z\right),\\
R_{k,2}:=&\left(-xy^{k-1},\,z^k -  x^k,\, y^{k-1}z\right),\\
R_{k,3}:=&\left(xz^{k-1},\,- yz^{k-1},\,x^k -  y^k\right).
\end{aligned}
\end{equation*} 

These syzygies satisfy the relation $xR_{k,1}+y R_{k,2}+zR_{k,3}=0$ and so that 
$C_{2 k}$ is a nearly
free divisor with exponents  $d_1=d_2=d_3=k$ and $\tau(C_{2 k})=3k(k-1)$.

To prove  \ref{nf-free-genus-3} we follow as in the proof
of Theorem~\ref{free-k-branches} \ref{free-k-branches-3}; the main difference is that $\pi_2$ has no ramification over
$C_2$ and in fact $C_4$ is the union of four lines in general position; their preimages.
If $k=2\ell$, since $\pi_{k}=\pi_\ell\circ\pi_2$, each irreducible component is a smooth
Fermat curve.

\end{proof}

These curves, for $k\geq 3$ odd, are of positive genus and give   
a counterexample to the nearly free part of Conjecture~\ref{q2}\ref{c1ii}
(with unbounded genus and number of singularities).
Furthermore, if $k\geq 5$, since $d_1=5\geq 4$ then by 
Corollary \ref{corC1} $C_{2k}$ is projectively rigid.
Note that it is not the case for $C_6$, where we find the dual to a smooth
cubic; in fact, not only this special dual is nearly free, an easy computation
gives the answer for the dual of a generic smooth cubic.

\section{Pencil associated to unicuspidal rational plane curves}\label{nearly-4-branches}

In this section we are going to show that it is possible to construct
a rational nearly free curve with singular points with more than three branches
(i.e., we do not need high genus curves).

Given a curve $C \subset \PP^2$, let $\pi: \widetilde \PP^2  \to \PP^2$
be the minimal,  (not the ``embedded'' minimal) resolution of singularities of $C$;
let $\widetilde C \subset \widetilde \PP^2$ be the strict transform of $C$,
and let ${\tilde \nu}(C)=\widetilde C  \cdot \widetilde C$ denote the self-intersection number of $\widetilde C$ on $\widetilde  \PP^2$.

\medskip

A {\it unicuspidal rational curve\/} is a pair $(C,P)$ where $C$ is a curve
and $P\in C$ satisfies $C \setminus \{P\} \cong \aff^1$.  We call $P$ the distinguished point of $C$.
Given a unicuspidal rational curve $C \subset \PP^2$ with singular point $P$,  in \cite{Dai_Mell_12,Dai_Mell_14},
 D. Daigle and the last named author,   
were interested in the unique pencil $\Lambda_C$ on $\PP^2$ satisfying $C \in \Lambda_C$
and $\Bs( \Lambda_C ) = \{P\}$
where $\Bs( \Lambda_C )$ denotes the base locus of $\Lambda_C$ on $\PP^2$. 

Let $\pi_ m: \widetilde \PP^2_m  \to \PP^2$ be the minimal resolution of the base points of the pencil.
By Bertini theorem, the singularities of the general  member $C_{\text{\rm gen}}$ of  $\Lambda_C$ are contained in    
$\Bs( \Lambda_C ) = \{P\}$.

For a unicuspidal rational curve $C \subset \PP^2$, we show
(cf.\ \cite[Theorem 4.1]{Dai_Mell_14})
that the general member of  $\Lambda_C$ is a rational curve if and only if $\tilde\nu(C) \ge 0$.
In such a case 

\begin{enumerate}
 \item the general element $C_{\text{\rm gen}}$ of  $\Lambda_C$  satisfies that  the   weighted cluster of infinitely near points
of $C_{\text{\rm gen}}$ and $C$    are equal (see \cite[Proposition 2.7]{Dai_Mell_12}).
\item $\Lambda_C$ has either $1$ or $2$ dicriticals,
and at least one of them has degree $1$.
\end{enumerate}

In view of these results, it is worth noting that
\emph{all currently known unicuspidal rational curves $C \subset \PP^2$ satisfy $\tilde\nu(C) \ge 0$},
see \cite[Remark~4.3]{Dai_Mell_14} for details.

Let $C \subset \PP^2$ be a unicuspidal rational curve of degree $d$ and with distinguished point $P$.
 In  in \cite[Proposition~1]{Dai_Mell_14} it is proved   $\Lambda_C$ is in fact  
 the set of effective divisors
$D$ of $\PP^2$ such that $\deg(D) = d$ and $i_P(C, D) \ge d^2$. 
The curve $C \in \Lambda_C$  because $i_P(C,C) = \infty > d^2$.

%
%
%

The main idea here is to take the general member $C_{\text{\rm gen}}$ of the pencil  $\Lambda_C$ 
for a nonnegative curve, i.e $\tilde\nu(C) \ge 0$. Doing this one gets a rational cuve  $C_{\text{\rm gen}}$ 
whose singularities is $\sing (C_{\text{\rm gen}})=\{P\}$ and the branches of  $C_{\text{\rm gen}}$ at $P$ is nothing else than the sum of the degrees
of the dicriticals.

The clasification of unicuspidal rational plane curve  with
$\bar{\kappa}(\PP^2\setminus C)=1$ has been started by Sh.~Tsunoda \cite{Tsu1} and finished by K. Tono  \cite{KeitaTono} 
(see also p. 125 in \cite{FLMN:London2006}).

Our next example starts  with $C_{49}$ with
$\bar{\kappa}(\PP^2\setminus C_{49})=1$. Secondly we take the pencil  $\Lambda_{C_{49}}$, and finally  its general member 
$C_{49, gen}$ has degree $49$ and is rational nearly-free with just one singular point which has 4 branches.

The curve $C_{49}$ is given by
\[f_{49}=((f_1^sy+\sum_{i=2}^{s+1}a_if_1^{s+1-i}x^{ia-a+1})^{a}-f_1^{as+1})/x^{a-1}=0,\]
where $f_1=x^{4-1}z+y^4$,  $a=4$, $s=3$, $a_2=\ldots=a_{s}\in \C$
and  $a_{s+1} \in \C\setminus\{0\} $. We can e.g.  $a_2=\ldots=a_{s}=0\in \C$
and  $a_{s+1}=1$. In this case, $d=a^2s+1= 49$, and the multiplicity sequence of $(C_{49},P)$ of the singular point $P:=[0,0,1]$
is $[36,12_{7},4_{6}]$. It is no-negative with $\tilde\nu(C_{49}) = 1$.

If we consider the rational curves $C_4$ defined by $f_1=0$ (resp. $C_{13}$ defined by 
$f_{13}:\,\, (f_1)^3 y+x^{13}=0$) then $i_P (C_{49}, C_4)=4 \cdot 49$ (resp.   $i_P (C_{49}, C_{13})=13 \cdot 49$).
Thus the curve $C_{13}C_{4}^{s(a-1)}$ belongs to the pencil $\Lambda_{C_{49}}$ if $s(a-1)=9$.

If we take the curve $C_{49,gen}$ defined by $f_{49,gen}:=f_{49}+13 f_{13} f_ 4^9=0$. 
This curve is irreducible, rational and $\sing (C_{49,gen})=\{P\}$ 
and the number of branches of $ C_{49,gen}$ at $P$ is $4$.

It is nearly free by using the computations with \texttt{Singular}~\cite{Sing}. 
A minimal  resolution \eqref{r2} for $f_{49,gen}$ 
is determined by  three syzygies of degrees $d_1=24$ and $d_2=d_3=25$. So that  $\mdr(f_{49,gen})=24$.
The computations yield a relation between these syzygies of multidegree $(2,1,1)$.
Then $C_{49,gen}$ is a rational nearly free curve. Let us note that
a direct computation using \texttt{Singular}~\cite{Sing} of the Tjurina number of the singular point of the curve fails,
but the \emph{nearly-free} condition makes the computation possible via Theorem~\ref{thm:dimca}\ref{thm1-dimca}:
$\tau(C_{49,gen})=(49-1)(49-24-1)+24^2-1=1727$ which is the result in \texttt{Singular} using characteristic 
$p=1666666649$.


\providecommand{\bysame}{\leavevmode\hbox to3em{\hrulefill}\thinspace}
\providecommand{\MR}{\relax\ifhmode\unskip\space\fi MR }
\providecommand{\MRhref}[2]{%
  \href{http://www.ams.org/mathscinet-getitem?mr=#1}{#2}
}
\providecommand{\href}[2]{#2}

\end{document}